\documentclass{amsart}

\usepackage{graphicx}
\usepackage{amssymb}
\usepackage{amsmath}
\usepackage{amsthm}
\newtheorem{theorem}{Theorem}[section]
\newtheorem{definition}{Definition}[section]
\newtheorem{lemma}{Lemma}[section]

\title{A Proof of the Strict Monotone 5-Step Conjecture}
\date{June 2018}
\author{J. Mackenzie Gallagher}

\email{jmgallagher36@gmail.com}

\author{Walter D. Morris, Jr.}
\address{Department of Mathematical Sciences,
George Mason University, Fairfax, Virginia, USA}
\email{wmorris@gmu.edu}

\begin{document}
\maketitle

\begin{abstract}
    A computer search through the oriented matroid programs with dimension 5 and 10 facets shows that the maximum strictly monotone diameter is 5.  Thus $\Delta_{sm}(5,10)=5$.  This enumeration is analogous to that of Bremner and Schewe for the non-monotone diameter of 6-polytopes with 12 facets.  Similar enumerations show that $\Delta_{sm}(4,9)=5$ and $\Delta_m(4,9)=\Delta_m(5,10)=6.$ We shorten the known non-computer proof of the strict monotone 4-step conjecture.
\end{abstract}

\section{Introduction}
The problem solved here for $d=5$ by computer enumeration was solved in the case $d=4$ by Holt and Klee \cite{HK}.  The method we use stems from Bremner and Schewe \cite{BremnerSchewe}. We follow the definitions and notation from these landmark papers.  

A linear functional is {\it admissible} for a convex polytope $P$ if it is not constant on any edge of $P$.  An admissible functional on a convex polytope $P$ induces an orientation of the graph of $P$, with edges oriented in the direction of increase of the functional.  For integers $d,n$ with $0<d<n$, three functions are of interest to us.  $\Delta(d,n)$ is the maximum diameter of a $d$-polytope with $n$ facets.  $\Delta_m(d,n)$ is the maximum length that the shortest directed path between two vertices of a $d$-polytope $P$ with $n$ facets can be, in an orientation given by an admissible functional.  $\Delta_{sm}(d,n)$ is the maximum length of a directed path from the source of a $d$-polytope $P$ with $n$ facets to its sink, where again the orientation is determined by an admissible functional. These functions satisfy $\Delta(d,n)\le\Delta_{sm}(d,n)\le\Delta_m(d,n)$.

Holt and Klee \cite{HK} showed that $\Delta_{sm}(4,8)=4,$ thereby proving the {\it strict monotone 4-step conjecture.}  Their proof (which we shorten in Lemma \ref{foureight}) involves some enumeration, but can be followed by a diligent reader.  We prove that $\Delta_{sm}(5,10)=5$ by computer enumeration.  Because $\Delta_{sm}(d,2d) \ge \Delta(d,2d)$ and Santos' \cite{Santos} result that $\Delta(43,86)\ge 44$, it is known that there is a largest $d$ for which $\Delta_{sm}(d,2d)=d$. This largest $d$ is not expected to be very large, due to the result of Todd \cite{Todd}, that $\Delta_m(4,8)>4$.  Holt and Klee gave an orientation of a $5$-polytope with 10 facets for which the shortest directed path from source to sink was greater than 5.  This orientation satisfied some combinatorial conditions that admissible orientations must satisfy, such as having a unique source and sink on every face.  

Bremner and Schewe \cite{BremnerSchewe}, \cite{BDHW} showed how to use satisfiability solvers to determine if there existed a counterexample to the 6-step conjecture $\Delta(6,12)=6$.  They showed that that not only did there not exist such a polytope, but there existed no {\it matroid polytope} that could serve as the dual of a combinatorial counterexample.  

We determined by computer that the orientation of Holt and Klee did not meet the more stringent combinatorial requirement that it come from an {\it oriented matroid program}.  This led us to adapt the technique of Bremner and Schewe to 
show that no oriented matroid program could serve as a counterexample to the strict monotone 5-step conjecture.  The satisfiability problems encountered are of roughly the same size as those appearing in Bremner and Schewe's investigation of the 6-step conjecture.  The computation showed that no counterexample exists.

\section {Oriented Matroid Programming}
The standard reference for oriented matroid programming is Chapter 10 of \cite{OMbook}. Historical development can be found in  \cite{Bland},\cite{EdmondsFukuda} and \cite{FL}.

Let $A$ be an $(n-d) \times n$ matrix of rank $n-d$.  In solving a linear program 

$$\mbox {Maximize } f=c^Tx+\beta \mbox { subject to }Ax=b, x\ge 0,$$
with the simplex method, one encounters simplex tableaux, such as the following: 

$$
T = \left[
\begin{array}{c|c|c|c}
I&N & 0& -\overline{b} \\ \hline
0& -\overline{c}^T& 1&-\overline{\beta}
\end{array}\right].
$$

The matrix $T$ is an $(n-d+1) \times (n+2)$ of rank $n-d+1$. The first $n$ columns correspond to the variables $x_1,x_2,\ldots,x_n$, and column $n+1$ corresponds to the variable $f$.  We make the following assumption: Every nonzero vector in the row space of $T$ has at least $d+2$ nonzero components. We can make this assumption because \cite{HK} shows that the monotone diameter is always maximized by a {\it simple} polytope.
This assumption also implies that every nonzero vector in the null space of $T$ has at least $n-d+2$ nonzero components.  The {\it support} of a vector $x$ in the row space of $T$ is the set of indices $i$ in $[n+2]$ for which $x_i \neq 0$. For a vector $x$ in the row space of $T$, the {\it sign vector} associated to $x$ is the vector $\sigma(x)$ in $\{-,0,+\}^{n+2}$ where each entry is the sign of the corresponding entry of $x$. Each such nonzero $\sigma(x)$ with minimal support is called a {\it circuit} of the {\it oriented matroid} realized by $T$. Sign vectors corresponding to nonzero vectors in the null space with minimal support are called {\it cocircuits} of the oriented matroid realized by $T$.  

The entries of $[n+2]$ indexing the columns of the identity matrix in the first $n-d$ rows of $T$ are called {\it basic}, and the entries indexing the columns of the matrix $N$ are called {\it nonbasic}. The {\it fundamental circuit} of the tableau $T$ is the sign vector of the bottom row.  The {\it fundamental cocircuit} is the sign vector for the vector in the null space of $T$ which is zero on all the nonbasic entries of $[n+2]$ and is positive in the last entry.   

A tableau corresponds to a vertex of the feasible region (is {\it feasible}) if the fundamental cocircuit is positive on all basic entries.  If the tableau $T$ above were feasible, then the fundamental cocircuit would be $(+,\ldots,+,0,\ldots,0,?,+)$, where the $+,\ldots,+$ are the basic entries, the $0,\ldots,0$ are the nonbasic entries, $?$ is an element of $\{-,+\}$ in the $n+1$ position. The vertex corresponding to a tableau is contained in the facets indexed by the nonbasic entries.  A nonbasic entry of the fundamental circuit of a feasible tableau is negative (resp. positive) if the objective function increases (resp. decreases) as one moves on the edge away from the corresponding facet.  Thus, if the tableau above corresponded to the source, the fundamental circuit would be $(0,\ldots,0,-,\ldots,-,+,?)$. If it were the sink, it would be 
$(0,\ldots,0,+,\ldots,+,+,?)$.  

The collection of all sign vectors of the null space of $T$ is called the set of {\it covectors} of an {\it oriented matroid} of {\it rank} $r=d+1$. The $(d+1)$-element subsets of $[n+2]$ are the {\it bases} of the oriented matroid.

\begin{theorem}\label{circchiro} (see Section 3.5 of \cite{OMbook}.)
 Given the assumptions above, there are exactly two functions $\chi$ from the set of ordered $(d+1)$-element subsets (bases) of $[n+2]$ to $\{+,-\}$ that satisfy the following:

\begin{enumerate}
    \item $\chi$ is alternating
    \item for any two ordered bases of the form $(e,x_2,\ldots,x_r)$ and $(f,x_2,\ldots, x_r)$, $e \neq f$, we have $\chi(f,x_2,\ldots,x_r)=-C(e)C(f)\chi(e,x_2,\ldots,x_r),$ where $C$ is one of the two opposite signed circuits of the oriented matroid with support $\{e,f,x_2,\ldots,x_r\}$.
    \item for any two ordered bases of the form $(e,x_2,\ldots,x_r)$ and $(f,x_2,\ldots, x_r)$, $e \neq f$, we have $\chi(f,x_2,\ldots,x_r)=-D(e)D(f)\chi(e,x_2,\ldots,x_r),$ where $D$ is one of the two opposite signed cocircuits of the oriented matroid with support $\{e,f\} \cup [n+2]\backslash\{x_2,\ldots,x_r\}$.
\end{enumerate}
\end{theorem}

\begin{definition}\label{chirodef} A mapping $\chi$ from the $r$-element ordered subsets of $[n+2]$ to $\{-1,+1\}$ is said to be a {\it uniform chirotope} of rank $r$ if it satisfies the following two conditions: 
\begin{enumerate}
    \item $\chi$ is alternating
    \item For all $\sigma \in {n \choose r-2}$ and all subsets $(x_1,\ldots,x_4) \subseteq [n+2]\backslash \sigma$,\newline $\{\chi(\sigma,x_1,x_2)\chi(\sigma,x_3,x_4), -\chi(\sigma,x_1,x_3)\chi(\sigma,x_2,x_4),\chi(\sigma,x_1,x_4)\chi(\sigma,x_2,x_3)\}$ \newline $=\{-1,+1\}$
\end{enumerate} 
\end{definition}

By Theorem 3.5.5 of \cite{OMbook}, each of the two functions obtained from the cocircuits and circuits of an oriented matroid must be a chirotope.  In order to determine if an orientation of a digraph is obtained from an admissible function on a polytope, we set up a satisfiability problem with a variable for each $(d+1)$-set. The alternating property implies that we only need a variable for the natural ordering of a given set.  Schewe \cite{Schewe} shows how to define clauses of a conjunctive normal form statement that are necessary to satisfy the second condition of Definition \ref{chirodef}. For each $r$-element subset $\sigma$ of $[n+2]$ and each 4-element subset of $[n+2]\backslash \sigma,$ there are 16 clauses with 6 literals each that must be satisfied.  We call these clauses the {\it chirotope constraints}. 
\subsection{Holt and Klee's Digraph}
To determine if the orientation of the 5-polytope proposed by Holt and Klee could come from an oriented matroid program, we listed the fundamental cocircuits, minus the last two entries, associated to the vertices as columns of the following matrix, which we call a {\it facet-vertex} matrix:

\begin{tiny}
\begin{verbatim}
[-1 -1  0 -1 -1 -1  0  0  0  0  0  0 -1  0  0  1  1 -1 -1  1  1  0  0  0  0  0  0 -1 -1  0  1 -1 -1  0  0  0  0  0  0 -1  0  0]
[-1  0  1  0  0  0  1  1  1  0  0  0  0 -1  0  1  1  0  0  0  0  1  1  1  1  0  0 -1  0 -1  0  0  0 -1  1  1  0  0  0  0 -1  0]
[-1 -1 -1  0  0  0 -1 -1 -1  1  1  1  0  0  0 -1  0  1  0  0  0  0  1  0  0  0  0  1 -1  1  0  0  0  1 -1 -1 -1  1  1  0  0  0]
[-1 -1 -1  1 -1 -1  0  0  0 -1 -1 -1  0  0  0  0  1  0  1  0  0  1  0  0  0  0  0  1  1 -1 -1  1  1  0  0  0  1 -1 -1  0  0  0]
[ 0  0  1  1  0  0 -1  0  0  1  0  0  1  0  1  0  0  0 -1  1  0 -1  0  1  0 -1 -1  0  0  1 -1  0  0 -1  0  0 -1  0  0 -1  0  1]
[ 0  1  0  0 -1  0  0  1  0  0 -1  0  0  1  1  0  0 -1  0  0 -1  0 -1  0 -1 -1 -1  0  1  0  0 -1  0  0 -1  0  0 -1  0  0  1  1]
[ 0  0  0  0  1 -1  0 -1  1  0  1 -1 -1 -1 -1  0  0  0  0 -1 -1  0  0  0  0 -1  0  0  0  0  0  1  1  0  1  1  0  1  1  1  1  1]
[ 0  0  0 -1  0  1  1  0 -1 -1  0  1 -1  1 -1  0  0  0  0  0  0  0  0 -1 -1  0  1  0  0  0  1  0  1  1  0  1  1  0  1  1  1  1]
[-1 -1 -1 -1 -1 -1 -1 -1 -1 -1 -1 -1 -1 -1 -1 -1 -1 -1 -1 -1 -1 -1 -1 -1 -1 -1 -1  0  0  0  0  0  0  0  0  0  0  0  0  0  0  0]
[ 0  0  0  0  0  0  0  0  0  0  0  0  0  0  0  1  1  1  1  1  1  1  1  1  1  1  1  1  1  1  1  1  1  1  1  1  1  1  1  1  1  1]
\end{verbatim}
\end{tiny}
Each row $i$ for $i$ in $[n]$ corresponds to a facet and each column corresponds to a vertex.  
The nonzero entries in a column $v$ correspond to the facets that contain vertex $v$.  An entry in position $(i,v)$ is $-1$ if 
the edge incident to vertex $v$ with only one end in facet $i$ is directed away from facet $i$.  If it is $+1$ then that edge is directed toward $v$.  

For example, column 3 of the matrix is $(0,1,-1,-1,1,0,0,0,-1,0)$. With the last two entries added, we have the circuit $(0,+,-,-,+,0,0,0,-,0,+,?)$ and cocircuit $(+,0,0,0,0,+,+,+,0,+,?,+)$ This implies, by Theorem \ref{circchiro}, the following conditions on $\chi:$ \newline
$\chi(2,3,4,5,9,1)=\chi(2,3,4,5,9,12),$\newline
$\chi(2,3,4,5,9,6)=\chi(2,3,4,5,9,12),$\newline
$\chi(2,3,4,5,9,7)=\chi(2,3,4,5,9,12),$\newline
$\chi(2,3,4,5,9,8)=\chi(2,3,4,5,9,12),$\newline
$\chi(2,3,4,5,9,10)=\chi(2,3,4,5,9,12),$\newline
from the cocircuit, and \newline
$\chi(2,3,4,5,9,12)=-\chi(11,3,4,5,9,12),$\newline
$\chi(2,3,4,5,9,12)=\chi(2,11,4,5,9,12),$\newline
$\chi(2,3,4,5,9,12)=\chi(2,3,11,5,9,12),$\newline
$\chi(2,3,4,5,9,12)=-\chi(2,3,4,11,9,12),$\newline
$\chi(2,3,4,5,9,12)=\chi(2,3,4,5,11,12),$\newline
from the circuit.  Each of these conditions can be enforced by two clauses with two literals each.  (See \cite{BremnerSchewe} for details).

We gathered the 20 clauses for each of the columns of the facet-vertex matrix, combined them with the chirotope constraints, and fed the resulting conjunctive normal form statement to the satisfiability solver SAT in cocalc.com.  

The satisfiability solver reported in less than a minute that there is no function $\chi$ satisfying both the chirotope conditions of Definition \ref{chirodef} and the conditions implied by the orientation.  The problem of completability of a partial chirotope is known to be NP-complete, see \cite{Tschirschnitz}.

\section{Realizable and Non-realizable Oriented Matroid Programs}
\subsection{Matroid Polytopes}
 Every oriented matroid program that we consider is defined by a uniform rank $r$ chirotope $\chi$ defined as in Definition \ref{chirodef} on the set $E=[n] \cup \{f,g\}$ with $r=d+1$. This chirotope yields, for each $n-d+2$ element subset $\underline D$ of $E$, two sign vectors with support $\underline D$ according to condition $3$ of Theorem \ref{circchiro}.  Each of the two sign vectors is the negative of the other.  The collection of all of these sign vectors is the set of cocircuits of the oriented matroid ${\mathcal M}$ defined by $\chi$.  

The cocircuits that are positive on element $g$, positive on a set $A$ of $n-d$ elements of $[n]$, and are equal to some nonzero sign on element $f$ are the {\it vertices} of the oriented matroid program defined by ${\mathcal M}$. We will refer to a vertex by the $d$-element set $[n]\backslash A$. By condition 3 of Theorem \ref{circchiro}, $\chi([n]\backslash A,i)=\chi([n]\backslash A,g)$ 
for each $i \in A$, where we assume $[n]\backslash A$ is ordered naturally.

If all of the cocircuits of ${\mathcal M}$ that are positive on an $(n-d)$-element subset of $[n]$ and contain $\{f,g\}$ in their support are also positive on element $g$, the oriented matroid program is said to be {\it bounded}. The collection of all cocircuits of ${\mathcal M}$ that are positive on an $(n-d)$-element subset of $[n]$ and contain $\{f,g\}$ in their support yields the set of vertices of a {\it matroid polytope.}  These can be thought of as the positive cocircuits of the oriented matroid ${\mathcal M}\backslash \{f,g\}$ of rank $d+1$ on $[n]$ obtained from ${\mathcal M}$ by deleting elements $f$ and $g$.  This is the structure used by \cite{BremnerSchewe} and \cite{BDHW} in studying the non-monotone diameters of polytopes. 

\begin{figure}[h]
\setlength{\unitlength}{0.10in} 
\centering 
\begin{picture}(65,25)

\put(14,5){\circle*{0.8}}
\put(24,5){\circle*{0.8}}
\put(14,15){\circle*{0.8}}\put(13,4){\makebox(0,0){v}}
\put(24,15){\circle*{0.8}}
\put(19,10){\circle*{0.8}}
\put(29,10){\circle*{0.8}}\put(30,21){\makebox(0,0){w}}
\put(19,20){\circle*{0.8}}
\put(29,20){\circle*{0.8}}

\thicklines
\put(14,5){\vector(1,0){5}}
\put(14,5){\vector(0,1){5}}
\put(14,5){\vector(1,1){3}}
\put(24,15){\vector(-1,0){5}}
\put(19,10){\vector(1,0){5}}
\put(29,10){\vector(-1,-1){3}}
\put(14,15){\vector(1,1){3}}
\put(19,20){\vector(0,-1){5}}
\put(24,5){\vector(0,1){5}}
\put(19,20){\vector(1,0){5}}
\put(24,15){\vector(1,1){3}}
\put(29,10){\vector(0,1){5}}

\put(19,5){\line(1,0){5}}
\put(14,10){\line(0,1){5}}
\put(17,8){\line(1,1){2}}
\put(19,15){\line(0,-1){5}}
\put(24,10){\line(1,0){5}}
\put(26,7){\line(-1,-1){2}}
\put(17,18){\line(1,1){2}}
\put(19,15){\line(-1,0){5}}
\put(24,10){\line(0,1){5}}
\put(24,20){\line(1,0){5}}
\put(27,18){\line(1,1){2}}
\put(29,15){\line(0,1){5}}
\end{picture}

\caption{Oriented matroid program with a monotone cycle} 
\label{fig:lnlblock} 
\end{figure}
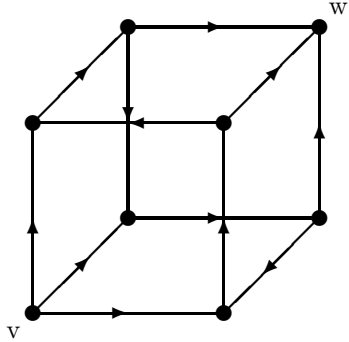

If one is not assuming boundedness of the oriented matroid program, one needs 
to keep the element $g$ and consider the positive cocircuits of the oriented matroid ${\mathcal M}\backslash \{f\}$. The vertices of the oriented matroid polytope ${\mathcal M}\backslash \{f,g\}$ corresponding to cocircuits of ${\mathcal M}\backslash \{f\}$ with $g$ negative (vertices ``beyond infinity") will not be vertices of the oriented matroid program.  In both the bounded and unbounded case, the positive cocircuits of ${\mathcal M}\backslash \{f\}$ containing $g$ define the undirected graph of the oriented matroid program.  

\subsection{Orientation of the Graph of the Matroid Polytope}

For every vertex $[n]\backslash A$ determined by an $(n-d)$-element subset $A$ of $[n]$, there is a unique circuit $C([n]\backslash A)$ (obtained from $\chi$ by condition 2 of Theorem \ref{circchiro}) with support $([n]\backslash A) \cup \{f,g\}$ that is positive on element $f$.  The set of these circuits determines the orientation of the digraph of the oriented matroid program. If a set $A' \subseteq [n]$ is such that $|A \cap A'|=n-d-1$ and $[n]\backslash A'$ is also a vertex of the oriented matroid program, then the arc containing the vertices $[n]\backslash A$
and $[n]\backslash A'$ leaves $[n]\backslash A$ if and only if the element $A'\backslash A$ is negative in $C([n]\backslash A)$.  Because the sign of $g$ in $C([n]\backslash A)$ is irrelevant to determining the orientation of the arcs incident to $[n]\backslash A,$ the orientation can be seen as being determined by the oriented matroid ${\mathcal M}/\{g\}$, obtained from ${\mathcal M}$ by contracting $g$.

An oriented matroid program obtained from a chirotope $\chi$ is {\it realizable} if its cocircuits are the sign vectors with minimal support in the null space of a matrix $T$.  An example of the digraph of an oriented matroid program that is not realizable appears in Figure 1.  Note that the six vertices other than the source $v$ and the sink $w$ form a directed cycle.  If the oriented matroid program were realizable, these vertices would give a cycle of nondegenerate pivots, which is not possible.

The orientation of Figure 1 has some properties one would expect from admissible functions on a polytope, such as a unique source and sink on every face.  On the other hand, none of the facets incident to vertex $w$ has its source adjacent to $v$.

We call an oriented matroid program {\it acyclic} if its digraph does not contain a directed cycle. We state the following lemma without proof.

\begin{lemma}
If an oriented matroid program is acyclic, vertex $v$ is the source of its digraph, and vertex $v'$ is adjacent to $v$ in some topological sweep of the digraph, then $v'$ is the source of the digraph of a facet not containing $v$.
\end{lemma}

We find it intriguing that for each of the values of $\Delta_{sm}$ and $\Delta_m$ that we have determined, the lower bound, which is attained by a realizable oriented matroid program, is matched by an upper bound that is obtained by searching for all oriented matroid program counterexamples, whether realizable or non-realizable. This happens despite the early stage at which non-realizable oriented matroid programs can be constructed.

For the rest of the paper, we will refer to statements proved by computer as Theorems and those proved without the use of computer as Lemmas.

\section{Bremner and Schewe's method}

If an orientation of a 5-polytope were a counterexample to the strictly monotone 5-step conjecture, then \cite{HK} shows that we could assume that the source and sink do not share any facets.  In particular, we could assume that the source of the orientation would be on facets $1,2,3,4,5$ and that its sink would be on facets $6,7,8,9,10$.  
Such a counterexample would lead to a function $\chi$ satisfying 10 constraints each for these two vertices.  The constraints
$$\chi((1,2,3,4,5),11)=\chi((1,2,3,4,5)_{k,12},11)$$ for $k=1,2,3,4,5$, together imply that $[1,2,3,4,5]$ is a vertex of the oriented matroid program, where $(1,2,3,4,5)_{k,12}$ is the ordered set $(1,2,3,4,5)$ with entry $k$ replaced by 12. Similarly, the constraints 
$$\chi((6,7,8,9,10),11)=\chi((6,7,8,9,10)_{k,12},11)$$ for $k=6,7,8,9,10$ imply that $[6,7,8,9,10]$ is a vertex.

The constraints $$\chi((1,2,3,4,5),12)=\chi((1,2,3,4,5)_{k,11},12),$$ for $k=1,2,3,4,5$, ensure that all of the edges containing vertex $[1,2,3,4,5]$ 
are oriented away from that vertex, while the constraints  $$\chi((6,7,8,9,10),12)=-\chi((6,7,8,9,10)_{k,11},12),$$ orient the edges containing $[6,7,8,9,10]$ toward $[6,7,8,9,10]$.  

We also 
have to introduce constraints to avoid each possible monotone path $v_0,v_1,\ldots,v_5$ from source to sink.  Each such path can be described by an ordered pair of permutations $(p,q)$, where $p$ is an ordering of the facets on the source and $q$ is an ordering of the facets on the sink.  The permutation $p$ indicates the order 
in which facets $1,2,3,4,5$ leave the nonbasis and $q$ indicates the order 
in which facets $6,7,8,9,10$ enter the nonbasis as one follows the path from $v_0$ to $v_5$.  We also assume that $\chi(1,2,3,4,5,6)=+1$.  The assumption that a particular pair $(p,q)$ yields a path from source to sink implies a conjunction of $\chi$ values for all the bases appearing in the conditions determined by the cocircuits for vertices on the path. This conjunction is described in detail in \cite{BremnerSchewe}.  For each $i$ from 1 to 4, we have to augment the conjunction with a $\chi$ value that determines the orientation of the edge $\{v_i,v_{i+1}\}$, in order to make the path monotone.  This conjunction is then negated, in order to ensure that the given length 5 monotone path does not appear.  The negation of the conjunction is a disjunction, which is added to the constraints given by the source and sink and the chirotope constraints.  

If we know that there is no $\chi$ satisfying the set of constraints we have described, then we would know there is no counterexample.  With generally available computing resources (solver SAT on cocalc.com), these constraints are quickly shown to be unsatisfiable for the $d=4$ analog of the problem.  The solver bogs down, however, when asked to verify the strictly monotone 5-step conjecture. 

Thus we need to divide the problem into parts.  We do this by specifying a path of length greater than 5 and forcing it to appear in the orientation.  The presence of the additional constraints determined by such a path lead the solver to determine unsatisfiability in 3 to 4 minutes.  

\begin{lemma} \label{acycle6}
$\Delta_{sm}(5,10)\le 6$.
\end{lemma}

\begin{proof}   Suppose that $L$ is a topological sweep of the acyclic digraph determined by an admissible objective function on a 5-dimensional polytope with 10 facets $1,2,\ldots,10$.  Suppose that source $v_0$, the first vertex in $L$, is on facets $1,2,3,4,5$ and the sink is on facets $6,7,8,9,10.$ Let $w$ be the first vertex in $L$ that is on two of the facets in $\{6,7,8,9,10\}$.  We can assume that $w$ is on the facets 6 and 7.  Then each of the vertices in the set $L_w$ of vertices in $L$ before $w$ is either $v_0$ or a neighbor of $v_0$. The arcs between $v_0$ and its neighbors are all oriented away from $v_0$, and there is an arc from a vertex in $L_w$ to $w$.  Thus there is a monotone path of length 2 from $v_0$ to $w$. Also, as in Lemma 3.1, $w$ is the source of the face $[6,7]$ that is the intersection of facets $6$ and $7$. The face $[6,7]$ is 3-dimensional with at most 8 facets. By \cite{HK} we can conclude that there is a monotone path of length at most $\lfloor\frac{2(8)}{3}\rfloor - 1 = 4$ from the source of face $[6,7]$ to its sink $[6,7,8,9,10]$. Combining these directed paths gives us a monotone path of length at most $2+4=6$ from $v_0$ to $[6,7,8,9,10]$.  
\end{proof}

\begin{theorem}\label{5step}
$\Delta_{sm}(5,10)=5$.
\end{theorem}
\begin{proof}
The paper \cite{BremnerSchewe} discusses the problem of listing representatives of all the combinatorial types of shortest paths of a given length between two fixed vertices of a polytope.  In view of Lemma \ref{acycle6}, we only need to look at paths of length 6.  Some guidelines make the listing easier.  
\begin{enumerate}
\item If a vertex $v_i$ is on a facet $F$ that did not contain $v_{i-1}$, then $v_{i+1}$ must also be on $F$. Otherwise, $v_{i-1},v_{i},v_{i+1}$ would all be on an edge.  
\item If a vertex $v_i$ is not on a facet $F$ that contained $v_{i-1},$ then $v_{i+1}$ can not be on $F$.  Otherwise, we could go from $v_{i-1}$ to $v_{i+1}$ in one step.  
\item The neighbors of the source may only appear as $v_1$ and the neighbors of the sink may only appear as $v_{k-1}$.
\end{enumerate}

The paths of length 6 are:
\newline
[[1,2,3,4,5],[6,2,3,4,5],[6,7,3,4,5],[8,7,3,4,5],[8,7,9,4,5],[8,7,9,6,5],[8,7,9,6,10]],\newline
[[1,2,3,4,5],[6,2,3,4,5],[6,7,3,4,5],[8,7,3,4,5],[8,7,9,4,5],[8,7,9,10,5],[8,7,9,10,6]],\newline
[[1,2,3,4,5],[6,2,3,4,5],[6,7,3,4,5],[6,7,8,4,5],[9,7,8,4,5],[9,7,8,10,5],[9,7,8,10,6]],\newline
[[1,2,3,4,5],[6,2,3,4,5],[6,7,3,4,5],[6,7,8,4,5],[6,9,8,4,5],[6,9,8,10,5],[6,9,8,10,7]],\newline
[[1,2,3,4,5],[6,2,3,4,5],[6,7,3,4,5],[6,7,1,4,5],[6,7,1,8,5],[6,7,9,8,5],[6,7,9,8,10]],\newline
[[1,2,3,4,5],[6,2,3,4,5],[6,7,3,4,5],[6,7,1,4,5],[6,7,1,8,5],[6,7,1,8,9],[6,7,10,8,9]],\newline
[[1,2,3,4,5],[6,2,3,4,5],[6,7,3,4,5],[6,7,8,4,5],[6,7,8,1,5],[6,7,8,1,9],[6,7,8,10,9]],\newline
[[1,2,3,4,5],[6,2,3,4,5],[6,7,3,4,5],[6,7,8,4,5],[6,7,8,2,5],[6,7,8,2,9],[6,7,8,10,9]].

The shortest monotone paths of length 6 involve only one revisit, that is, there is only one facet that both enters and leaves vertices of the sequence. The enumeration of these is analogous to that in \cite{BremnerSchewe}.  As noted there, there is a correspondence between the first 4 paths, for which one of the facets $6,7$ was revisited, and the last 4 in which one of the facets $1,2$ was revisited.  The last 4 may be thought of as reversing the first 4, so one does not need to list them.  

For each of the paths of length 6 that we added to the conditions, the additional constraints allowed the solver to determine unsatisfiability in 3 to 4 minutes.  \end{proof}
\section{Related results}
\subsection{$\Delta_{sm}(5,10)=5$ for general oriented matroid programs}

Lemma \ref{acycle6} assumed acyclicity of the oriented matroid digraph, which holds in the realizable case but not in general.  The proof of Theorem \ref{5step} shows that there is no oriented matroid program counterexample with shortest path of length 6.  One still has to account for examples that are not acyclic and have shortest paths of length greater than 6.  We shall first show that it is sufficient to exclude examples for which the shortest monotone path is of length 7.

\begin{lemma}\label{delmfiveten} 
$\Delta_{m}(5,10)\le 7$.
\end{lemma}

\begin{proof} Suppose that vertex $v_0$ is on facets $1,2,3,4,5$ and the sink is on facets $6,7,8,9,10,$ and the path $P=(v_0,v_1,\ldots,v_k)$ is a shortest monotone path from $v_0$ to the sink.  We do not assume that $v_0$ is the source of the orientation.  We can assume without loss of generality that $v_1$ is on facets $6,2,3,4,5.$ If $v_2$ is not on facet 6 or $v_2$ is on facet 1, then $v_2$ is a vertex adjacent to $v_0$ and could have been reached in one step rather than 2.  This contradicts the minimality of the path $P$.  Thus we can assume without loss of generality that $v_2$ is on facets 6 and 7, so that it shares a 3-dimensional face with the sink.  It was proved in \cite{Kl2} that $\Delta_m(3,n) \le n-3$ for all $n$. Thus there is a monotone path of length at most $8-3=5$ from $v_2$ to the sink. \end{proof}

Lemma \ref{delmfiveten} shows that a linear program with shortest monotone path of length 7 from $[1,2,3,4,5]$ to $[6,7,8,9,10]$ will have a path that starts with moving on to two of the facets in $\{6,7,8,9,10\}$, such as starting with $[1,2,3,4,5]$, $[6,2,3,4,5]$, $[6,7,3,4,5]$, and staying on facets $6$ and $7$ for the rest of the path.  Klee \cite{Kl2} shows that this path can be assumed to be {\it nonrevisiting}.  That is, the path stays on $6,7$ and the only facets that both enter and leave on the way from $[6,7,3,4,5]$ to $[6,7,8,9,10]$ are facets $1$ and $2$.  A closer look the proof of this theorem of \cite{Kl2} shows that its arguments can be directly applied to oriented matroid programs. The main reasons for this are that the undirected graphs of 3-dimensional matroid polytopes are realizable and directed graphs of 2-dimensional oriented matroid programs are realizable.  Thus we can assume that there is a shortest monotone path of length $7$ of one of the following types, starting with $[1,2,3,4,5],[6,2,3,4,5],[6,7,3,4,5]$: 
\newline
[[6,7,3,4,5],[6,7,8,4,5],[6,7,8,1,5],[6,7,8,1,2],[6,7,8,9,2],[6,7,8,9,10]],\newline
[[6,7,3,4,5],[6,7,8,4,5],[6,7,8,2,5],[6,7,8,2,1],[6,7,8,9,1],[6,7,8,9,10]],\newline
[[6,7,3,4,5],[6,7,1,4,5],[6,7,1,8,5],[6,7,1,8,2],[6,7,9,8,2],[6,7,9,8,10]],\newline
[[6,7,3,4,5],[6,7,1,4,5],[6,7,1,8,5],[6,7,2,8,5],[6,7,2,8,9],[6,7,10,8,9]],\newline
[[6,7,3,4,5],[6,7,1,4,5],[6,7,1,2,5],[6,7,1,2,8],[6,7,9,2,8],[6,7,9,10,8]],\newline
[[6,7,3,4,5],[6,7,1,4,5],[6,7,1,2,5],[6,7,8,2,5],[6,7,8,2,9],[6,7,8,10,9]],\newline
[[6,7,3,4,5],[6,7,1,4,5],[6,7,1,2,5],[6,7,1,2,8],[6,7,1,9,8],[6,7,10,9,8]],\newline
[[6,7,3,4,5],[6,7,1,4,5],[6,7,1,2,5],[6,7,8,2,5],[6,7,8,9,5],[6,7,8,9,10]]

For each of these paths, enforcing its presence, together with the chirotope clauses and the clauses that excluded paths of length 5, led the satisfiability solver to determine in 3 to 4 minutes that the instance was not satisfiable.
\subsection{Monotone Paths for Dimension 5}
Here we enforce the condition that $[6,7,8,9,10]$ is a sink, but for $[1,2,3,4,5]$ we only require that it be a vertex and do not specify the orientation of the edges incident to it.  This leaves open the possibility that the oriented matroid programs generated by the solver could be unbounded. 

\begin{theorem}
$\Delta_m(5,10)= 6$
\end{theorem}
By standard methods, such as making a prism over Todd's example which shows 
$\Delta_m(4,8)\ge 5$, one can show that $\Delta_m(5,10)\ge 6.$  

Showing that $\Delta_m(5,10)\le 6$ pushes the limits of our computational resources.  One must exclude all monotone paths of length 5 and 6 from $[1,2,3,4,5]$ to $[6,7,8,9,10]$, and then enforce each of the paths of length 7 listed above.  The $8(5!)^2$ clauses to exclude paths of length 6 slowed down the computation considerably. The program did show that there was no oriented matroid program with shortest monotone path from $[1,2,3,4,5]$ to $[6,7,8,9,10]$ of length 7.  This includes nonrealizable oriented matroid programs as well as realizable ones, and programs with bounded as well as unbounded feasible regions.

\subsection{Dimension 4 with 8 facets}
\begin{lemma}
\label{foureight} $\Delta_{sm}(4,8)=4$.  
\end{lemma}

\begin{proof} This is a shortening of the proof given in \cite{HK}. There it was shown that $\Delta_{sm}(3,7) = \lfloor\frac{2}{3}(7)\rfloor -1 = 3.$  We are assuming that the oriented matroid program is realizable, therefore acyclic.  We can assume that we are looking for a monotone path from $[1,2,3,4]$ to $[5,6,7,8],$ where $[1,2,3,4]$ is the source and $[5,6,7,8]$ is the sink. Consider a vertex $v$ which is adjacent to the source in a topological sweep of the digraph.  Then $v$ is a neighbor of $[1,2,3,4]$, say $v= [5,2,3,4]$.  The vertex $v$ is then the source of facet $5$, by Lemma 3.1, so there is a monotone path from $v$ to $[5,6,7,8]$ (the sink of facet $5$) of length 3.  Thus we have a monotone path of length $4$ from $[1,2,3,4]$ to $[5,6,7,8]$.  
\end{proof}

We must point out that in the general setting of oriented matroid programming, we cannot assume that the source $[1,2,3,4]$ is adjacent to one of the facet sources from the facets on $[5,6,7,8].$  The best proof for the oriented matroid programming version of $\Delta_{sm}(4,8)=4$ remains that of \cite{HK}.

In order to express quantitatively how close $\Delta_{sm}(4,8)$ is to being 5, we had the computer search for examples with the shortest monotone path from $[1,2,3,4]$ to the sink $[5,6,7,8]$ of length 5, but with all possible placements of the source other than $[1,2,3,4]$, and all possible outmaps for the vertex $[1,2,3,4]$.  

\begin{theorem}\label{wheresource}
Suppose that the digraph contains a shortest path  of length 5 from  $[1,2,3,4]$ to $[5,6,7,8]$ and that $[5,6,7,8]$ is the sink of the orientation.    Each of the possible placements for the source shares at least two of the facets $5,6,7,8$ with the sink.  The  size of the outmap for $[1,2,3,4]$ is 2.
\end{theorem}

We see from these computer results that the examples that yield shortest paths of length 5 are far from having the source at $[1,2,3,4].$

\subsection{Dimension 4 with 9 facets}
\begin{lemma}
$\Delta_{sm}(4,9)\le 6$
\end{lemma}
\begin{proof}

 Suppose that $[1,2,3,4]$ is the source and $[6,7,8,9]$ is the sink of an orientation.  If each of the neighbors of $[1,2,3,4]$ were on facet 5, then the polytope would be a simplex with the five facets $1,2,3,4,5$, contradicting the assumption that $[6,7,8,9]$ is the sink.  Thus we can assume that $v_1=[6,2,3,4]$ is a vertex.  Unfortunately, we can not assume that $v_1$ is the source of facet 6.  Facet 6 is a $3$-polytope with at most 8 facets, so by Klee \cite{Kl2} there is a nonrevisiting directed path of length at most $8-3=5$ from $[6,2,3,4]$ to $[6,7,8,9]$. Thus $\Delta_{sm}(4,9) \le 5+1 = 6.$ \end{proof}

The nonrevisiting path on facet 6 is one of the following types: \newline
[[6, 2, 3, 4], [6, 5, 3, 4], [6, 5, 1, 4], [6, 5, 1, 7], [6, 5, 8, 7], [6, 9, 8, 7]],\newline
[[6, 2, 3, 4], [6, 5, 3, 4], [6, 5, 1, 4], [6, 5, 1, 7], [6, 8, 1, 7], [6, 8, 9, 7]],\newline
[[6, 2, 3, 4], [6, 5, 3, 4], [6, 5, 1, 4], [6, 7, 1, 4], [6, 7, 8, 4], [6, 7, 8, 9]],\newline
[[6, 2, 3, 4], [6, 5, 3, 4], [6, 5, 1, 4], [6, 7, 1, 4], [6, 7, 1, 8], [6, 7, 9, 8]],\newline
[[6, 2, 3, 4], [6, 5, 3, 4], [6, 5, 7, 4], [6, 5, 7, 1], [6, 8, 7, 1], [6, 8, 7, 9]],\newline
[[6, 2, 3, 4], [6, 5, 3, 4], [6, 5, 7, 4], [6, 1, 7, 4], [6, 1, 7, 8], [6, 9, 7, 8]],\newline
[[6, 2, 3, 4], [6, 7, 3, 4], [6, 7, 1, 4], [6, 7, 1, 5], [6, 7, 8, 5], [6, 7, 8, 9]],\newline
[[6, 2, 3, 4], [6, 7, 3, 4], [6, 7, 5, 4], [6, 7, 5, 1], [6, 7, 8, 1], [6, 7, 8, 9]].

\begin{theorem}
$\Delta_{sm}(4,9)=5.$
\end{theorem}
\begin{proof} Klee and Walkup showed that $\Delta(4,9)= 5.$ As pointed out in \cite{HK}, there is a projective image of a polytope realizing this diameter and a monotone linear functional so that the images of the vertices of the original polytope at distance 5 are sent to the source and sink of the image.  This proves $\Delta(d,n)\le \Delta_{sm}(d,n)$ in general.  Thus $\Delta_{sm}(4,9)\ge 5.$  We have a computer proof for the inequality $\Delta_{sm}(4,9)\le 5.$
In addition to excluding paths of length 4, this computation excludes the $8(4!)^2$ paths of length 5 of the following types:\newline [[1, 2, 3, 4], [5, 2, 3, 4], [5, 6, 3, 4], [5, 6, 7, 4], [5, 6, 7, 8], [9, 6, 7, 8]],\newline
[[1, 2, 3, 4], [5, 2, 3, 4], [5, 6, 3, 4], [5, 6, 7, 4], [8, 6, 7, 4], [8, 6, 7, 9]],\newline
[[1, 2, 3, 4], [5, 2, 3, 4], [5, 6, 3, 4], [7, 6, 3, 4], [7, 6, 8, 4], [7, 6, 8, 9]],\newline
[[1, 2, 3, 4], [6, 2, 3, 4], [6, 5, 3, 4], [6, 5, 7, 4], [6, 5, 7, 8], [6, 9, 7, 8]],\newline
[[1, 2, 3, 4], [6, 2, 3, 4], [6, 5, 3, 4], [6, 5,  7, 4], [6, 8, 7, 4], [6, 8, 7, 9]],\newline
[[1, 2, 3, 4], [6, 2, 3, 4], [6, 7, 3, 4], [6, 7, 5, 4], [6, 7, 5, 8], [6, 7, 9, 8]],\newline
[[1, 2, 3, 4], [6, 2, 3, 4], [6, 7, 3, 4], [8, 7, 3, 4], [8, 7, 9, 4], [8, 7, 9, 6]],\newline
[[1, 2, 3, 4], [6, 2, 3, 4], [6, 7, 3, 4], [6, 7, 1, 4], [6, 7, 1, 8], [6, 7, 9, 8]].

Then each of the eight paths of length 6 above is enforced, in each case leading quickly to the conclusion that no such example is possible.\end{proof} 

The following facet-vertex matrix comes from an oriented matroid program that has a shortest path from $[1,2,3,4]$ to $[6,7,8,9]$ of length 6, but has a source that is adjacent to $[1,2,3,4]$.

\begin{footnotesize}
\begin{verbatim}
[-1 -1 -1 -1  1 -1 -1  1 -1 -1  0  0  0  0  0  0  0  0  0  0  0]
[-1 -1 -1 -1  0  0  0  0  0  0 -1 -1  1  1 -1 -1 -1 -1  0  0  0]
[ 1 -1  0  0  1  1 -1  0  0  0  1  1 -1 -1  0  0  0  0 -1  0  0]
[ 1  0  1  0 -1  0  0 -1  0  0 -1  0  0  0  0  0  0  0  1  0  0]
[ 0  0  0  0  1  1  0  1  1  0  0  1  1  0  1  1  0  0  1 -1  0]
[ 0  0 -1  1  0  0  0 -1 -1  1  1 -1  0  0 -1  0  1  0 -1  1  1]
[ 0  0  0  0  0  1  1  0  1  1  0  0  1  1  0  1  0 -1  0  1  1]
[ 0  0  0  0  0  0  0  0  0  0  0  0  0  0 -1  1  1  1  0  1  1]
[ 0 -1  0 -1  0  0 -1  0  0 -1  0  0  0 -1  0  0 -1 -1  0  0  1]
\end{verbatim}
\end{footnotesize}
\begin{theorem}
$\Delta_m(4,9)=6$
\end{theorem}
Examples showing $\Delta_m(4,9)\ge 6$ can be constructed from those showing 
$\Delta_m(4,8)\ge 5$ by adding an extra inequality.
The proof for $\Delta_{sm}(4,9)\le 6$ given above does not carry over to the
$\Delta_m$ case.  This is because the neighbors of $[1,2,3,4]$ on arcs leaving
$[1,2,3,4]$ might all be on facet 5, while the neighbors of $[1,2,3,4]$ on
facets in $\{6,7,8,9\}$ could all be on arcs directed toward $[1,2,3,4]$.  
Thus we need a computer proof.  
Due to Proposition 1.1 of \cite{HK}, our computation that showed
$\Delta_m(5,10)\le 6$ also implies $\Delta_m(4,9)\le 6$.

\subsection{Experimenting with (5,10)}

We would like to contrast Theorem \ref{wheresource}, obtained for the case $(4,8)$, with similar observations for the $(5,10)$ case that indicate that the strictly monotone five-step conjecture is almost false. 

\begin{theorem}
For each of the following properties, there exist oriented matroid programs of dimension 5 with 10 facets, with the sink at $[6,7,8,9,10]$ and with the shortest directed path from $[1,2,3,4,5]$ to $[6,7,8,9,10]$ of length 6, having the given property:
\begin{enumerate}
    \item The outmap of vertex $[1,2,3,4,5]$ has size 4,
    \item The source is a neighbor of $[1,2,3,4,5]$.
\end{enumerate}
\end{theorem}
Here are facet-vertex matrices for these examples:
\begin{tiny}
\begin{verbatim}
[-1 -1 -1 -1  1 -1  1 -1 -1  1 -1  1 -1 -1 -1 -1 -1 -1 -1 -1 -1 -1  0  0  0  0  0  0  0  0  0  0  0  0  0  0  0  0  0  0  0  0]
[ 1  1  1  1 -1  1 -1 -1  1 -1  1 -1  0  0  0  0  0  0  0  0  0  0 -1  1  1 -1 -1 -1  1  1  0  0  0  0  0  0  0  0  0  0  0  0]
[-1 -1 -1 -1  1 -1  0  0  0  0  0  0 -1  1 -1 -1  0  0  0  0  0  0 -1 -1  1  1  1  1  0  0 -1 -1  1  1  1 -1 -1  0  0  0  0  0]
[-1 -1  0  0  0  0  1  1  0  0  0  0  1  0  0  0  1  0  0  0  0  0  1  1 -1  1  0  0 -1 -1  1  1 -1 -1 -1  0  0  1 -1  0  0  0]
[-1  0 -1  0  0  0 -1  0 -1 -1  1  0 -1  1  0  0  1  1  1  1  0  0 -1  0  0  0 -1  1 -1  0 -1 -1  1  1  0 -1  0 -1 -1 -1 -1  0]
[ 0  0  0  0  0  0  0  0  1 -1  0 -1  0  0  0  0  0 -1 -1 -1  1  1  1 -1  0  0  1  0 -1 -1 -1  0  0  0  0  0  0  1  0  1  1  1]
[ 0  0  0  1 -1  0  0  0 -1  0  1  1  0  0 -1 -1  0  1  0  0 -1 -1  0  1 -1  0 -1  1  0  1  1 -1  0  0 -1  0  1 -1 -1 -1  0  1]
[ 0  0  1 -1  0  0  0  0  0  0 -1  0  0 -1 -1  0  0 -1 -1  0 -1  0  0  0  0  0  0 -1  0  0  0  1 -1  0 -1  1  1  0 -1 -1  1  1]
[ 0  0  0  0  1  1  1  1  0  1  0  1  0  0  0  1  1  0  0  1  0  1  0  0  1  1  0  0  1  1  0  0  1  1  1  1  1  1  1  1  1  1]
[ 0  1  0  0  0  1  0  1  0  0  0  0 -1 -1 -1  1  1  0 -1  1 -1  1  0  0  0  1  0  0  0  0  0  0  0  1  0  1  1  0  0  0  1  1]
\end{verbatim}
\end{tiny}

\begin{tiny}
\begin{verbatim}
[-1 -1 -1 -1 -1 -1 -1 -1 -1 -1  1  1  1 -1 -1 -1  1  1  1  1  1  1 -1 -1  0  0  0  0  0  0  0  0  0  0  0  0  0  0]
[ 1 -1  1 -1 -1  1 -1 -1  0  0  0  0  0  0  0  0  0  0  0  0  0  0  0  0 -1 -1 -1  0  0  0  0  0  0  0  0  0  0  0]
[-1 -1 -1 -1  0  0  0  0 -1 -1  1  1  1 -1 -1 -1  0  0  0  0  0  0  0  0  1  1  0 -1 -1 -1  1 -1 -1 -1  0  0  0  0]
[ 1  1  0  0  1 -1  0  0 -1  1 -1 -1 -1  0  0  0  1  1 -1  0  0  0  0  0  1  0  1  1  1  1 -1  0  0  0 -1 -1  0  0]
[-1  0 -1  0  1  0  1  1 -1  0  0  0  0  0  0  0 -1 -1 -1 -1 -1 -1  0  0  1  1  1 -1 -1  0  0 -1  0  0  1 -1 -1  0]
[ 0  1  0  1  0 -1  1 -1  0 -1  0  0  0 -1  0  0  0  0  0 -1 -1  1  1  1  1  1  1 -1  0 -1  0 -1 -1  0 -1  0  1  1]
[ 0  0  0  0  0  0  0  0  0  0 -1  1  0  0 -1  0  1  1  0  1  1  0 -1  0  0  0  0  1  1  1 -1  1  1  1  1 -1 -1  1]
[ 0  0 -1 -1  0  0 -1  0 -1  0  0  0  1 -1  0  1  0  0 -1  0  0 -1  0  1  0 -1  0  0  1  0 -1 -1 -1  1  0 -1 -1  1]
[ 0  0  0  0  0  0  0  0  0  0  1  0  1  0  1  1  1  0  1 -1  0  1  1  1  0  0  0  0  0  0  1  0  0  1  0  1  1  1]
[ 0  0  0  0  1  1  0  1  0  1  0 -1  0  1  1  1  0 -1  0  0  1  0  1  1  0  0 -1  0  0  1  0  0  1  1  1  0  0  1]
\end{verbatim}
\end{tiny}

We did not test these oriented matroid programs for realizability. Because we found many such examples, we conjecture that there are realizable instances.  For the current state of realizability testing methods, see \cite{Firsching}.  

\section{Conclusions}
The computer enumerations that showed $\Delta_{sm}(5,10)=\Delta_{sm}(4,9)=5$ lead us to suspect that there exist reasonably short proofs of these facts.  
Perhaps this is an illusion.  We hope to have convinced the reader that oriented matroid programming is a useful tool for discovering properties of polytopal digraphs.

\end{document}